\documentclass[11pt,letterpaper]{amsart}
\usepackage{cases}
\usepackage{mathrsfs}
\usepackage{color}
\usepackage{latexsym,amssymb}
\usepackage{a4wide}
\usepackage{amsthm}
\usepackage{amsmath}
\usepackage{graphicx}
\input xy
\xyoption{all}
\usepackage{verbatim}
\usepackage[lite,abbrev,msc-links]{amsrefs}
\usepackage{framed}
\usepackage{bm}

\numberwithin{equation}{section}
\begin{document}
	\theoremstyle{plain}
	\newtheorem{thm}{Theorem}[section]
	\newtheorem{lem}[thm]{Lemma}
	\newtheorem{cor}[thm]{Corollary}
	\newtheorem{cor*}[thm]{Corollary*}
	\newtheorem{prop}[thm]{Proposition}
	\newtheorem{prop*}[thm]{Proposition*}
	\newtheorem{conj}[thm]{Conjecture}
	\theoremstyle{definition}
	\newtheorem{construction}{Construction}
	\newtheorem{notations}[thm]{Notations}
	\newtheorem{question}[thm]{Question}
	\newtheorem{prob}[thm]{Problem}
	\newtheorem{rmk}[thm]{Remark}
	\newtheorem{remarks}[thm]{Remarks}
	\newtheorem{defn}[thm]{Definition}
	\newtheorem{claim}[thm]{Claim}
	\newtheorem{assumption}[thm]{Assumption}
	\newtheorem{assumptions}[thm]{Assumptions}
	\newtheorem{properties}[thm]{Properties}
	\newtheorem{exmp}[thm]{Example}
	\newtheorem{comments}[thm]{Comments}
	\newtheorem{blank}[thm]{}
	\newtheorem{observation}[thm]{Observation}
	\newtheorem{defn-thm}[thm]{Definition-Theorem}
	\newtheorem*{Setting}{Setting}

	\newcommand{\sA}{\mathscr{A}}
	\newcommand{\sB}{\mathscr{B}}
	\newcommand{\sC}{\mathscr{C}}
	\newcommand{\sD}{\mathscr{D}}
	\newcommand{\sE}{\mathscr{E}}
	\newcommand{\sF}{\mathscr{F}}
	\newcommand{\sG}{\mathscr{G}}
	\newcommand{\sH}{\mathscr{H}}
	\newcommand{\sI}{\mathscr{I}}
	\newcommand{\sJ}{\mathscr{J}}
	\newcommand{\sK}{\mathscr{K}}
	\newcommand{\sL}{\mathscr{L}}
	\newcommand{\sM}{\mathscr{M}}
	\newcommand{\sN}{\mathscr{N}}
	\newcommand{\sO}{\mathscr{O}}
	\newcommand{\sP}{\mathscr{P}}
	\newcommand{\sQ}{\mathscr{Q}}
	\newcommand{\sR}{\mathscr{R}}
	\newcommand{\sS}{\mathscr{S}}
	\newcommand{\sT}{\mathscr{T}}
	\newcommand{\sU}{\mathscr{U}}
	\newcommand{\sV}{\mathscr{V}}
	\newcommand{\sW}{\mathscr{W}}
	\newcommand{\sX}{\mathscr{X}}
	\newcommand{\sY}{\mathscr{Y}}
	\newcommand{\sZ}{\mathscr{Z}}
	\newcommand{\bZ}{\mathbb{Z}}
	\newcommand{\bN}{\mathbb{N}}
	\newcommand{\bQ}{\mathbb{Q}}
	\newcommand{\bC}{\mathbb{C}}
	\newcommand{\bR}{\mathbb{R}}
	\newcommand{\bH}{\mathbb{H}}
	\newcommand{\bD}{\mathbb{D}}
	\newcommand{\bE}{\mathbb{E}}
	\newcommand{\bP}{\mathbb{P}}
	\newcommand{\bV}{\mathbb{V}}
	\newcommand{\cV}{\mathcal{V}}
	\newcommand{\cF}{\mathcal{F}}
	\newcommand{\cI}{\mathcal{I}}
	\newcommand{\cW}{\mathcal{W}}
	\newcommand{\cU}{\mathcal{U}}
	\newcommand{\bfM}{\mathbf{M}}
	\newcommand{\bfN}{\mathbf{N}}
	\newcommand{\bfX}{\mathbf{X}}
	\newcommand{\bfY}{\mathbf{Y}}
	\newcommand{\spec}{\textrm{Spec}}
	\newcommand{\dbar}{\bar{\partial}}
	\newcommand{\ddbar}{\partial\bar{\partial}}
	\newcommand{\redref}{{\color{red}ref}}
	
	\title[Stratified Hyperbolicity of the moduli stack of stable minimal models] {Stratified Hyperbolicity of the moduli stack of stable minimal models, II: Big Picard Theorem and the stratified Brody hyperbolicity}
	
	\author[Junchao Shentu]{Junchao Shentu}
	\email{stjc@ustc.edu.cn}
	\address{School of Mathematical Sciences,
		University of Science and Technology of China, Hefei, 230026, China}

	\begin{abstract}
		This is the second paper in the series on the global geometry of Birkar's moduli space of stable minimal models (e.g., the KSBA moduli stack). We introduce a birationally admissible stratification of the Deligne-Mumford stack of stable minimal models, such that the universal family over each stratum admits a simple normal crossing log birational model. The main result of this paper is to show that each stratum $\sS$ is Picard hyperbolic, Borel hyperbolic, and Brody hyperbolic in a certain sense.
	\end{abstract}
	
	\maketitle
	
	\section{Introduction}
	This is the second paper in a series of studies that investigates the global geometry of C. Birkar's moduli stack of stable minimal models and the KSBA moduli stack, building upon the results presented in \cite{stjc2025}. For $d \in \bN$, $c \in \bQ^{\geq 0}$, a finite set $\Gamma \subset \bQ^{> 0}$, and $\sigma \in \bQ[t]$, let $\sM_{\rm slc}(d,\Phi_c,\Gamma,\sigma)$ denote the Deligne-Mumford stack of $(d,\Phi_c,\Gamma,\sigma)$-stable minimal models as introduced by Birkar in \cite{Birkar2022}. Additionally, let $\Delta := \{z \in \bC \mid |z| < 1\}$ represent the open unit disk in the complex plane, and let $\Delta^\ast := \Delta \setminus \{0\}$ denote the punctured unit disk.
	The main result of this paper is the following theorem:  
	\begin{thm}\label{thm_main}  
		Let $f^o:(X^o,B^o),A^o \to S^o$ be a strictly birationally admissible family of $(d,\Phi_c,\Gamma,\sigma)$-stable minimal models over a quasi-projective variety $S^o$, defining a quasi-finite classifying map $\xi^o:S^o \to \sM_{\rm slc}(d,\Phi_c,\Gamma,\sigma)$. Then the following statements hold:  
		\begin{itemize}  
			\item \emph{(Big Picard Theorem)} Let $S$ be a projective variety containing $S^o$ as a Zariski open subset. Then $(S, S \setminus S^o)$ forms a Picard pair, meaning that for any holomorphic map $\gamma: \Delta^\ast \to S$ from the punctured unit disk $\Delta^\ast$, either $0 \in \overline{\gamma^{-1}(S \setminus S^o)}$, or $\gamma$ can be extended to a holomorphic map $\overline{\gamma}: \Delta \to S$. This generalizes the classical Big Picard theorem, which asserts that $(\bP^1, \{0, 1, \infty\})$ is a Picard pair.  
			\item \emph{(Borel Hyperbolicity)} Any holomorphic map from a finite type reduced scheme to $S^o$ is algebraic.  
			\item \emph{(Brody Hyperbolicity)} $S^o$ is Brody hyperbolic, i.e., there exists no non-constant holomorphic map $\bC \to S^o$.  
		\end{itemize}  
	\end{thm}  
	A family of stable minimal models $(X^o, B^o), A^o \to S^o$ is said to be \emph{strictly birationally admissible} if $(X^o, A^o + B^o)$ admits a simultaneous simple normal crossing log birational model (see Definition \ref{defn_admissible}). More formally, the strictly birationally admissible condition ensures that the family does not contain degenerate fibers within the framework of birational geometry. This condition plays a crucial role in Theorem \ref{thm_main}, as demonstrated in \cite{stjc2025}.
	
	Given a functorial semi-log desingularization procedure (see, e.g., \cite{Bierstone2013} and \cite[\S 10.4]{Kollar2013}), one can associate with $\sM_{\rm slc}(d,\Phi_c,\Gamma,\sigma)$ a stratification such that over each stratum $\sS$, the universal family is strictly birationally admissible. Such a stratification is referred to as a \emph{birationally admissible stratification} (cf. Section \ref{section_stack}). A birationally admissible stratification always exists but is typically not unique, as it depends on the choice of the functorial desingularization functor. A typical example arises in the KSBA compactification of the moduli space of principally polarized manifolds. In this case, the open substack parameterizing smooth fibers constitutes a birationally admissible stratum, while the birationally admissible stratification on the boundary is determined by the semi-log desingularization process.    
	
	A direct consequence of Theorem \ref{thm_main} is the following result.  
	\begin{thm}\label{thm_main_moduli}  
		Let $\sS$ be a stratum of a birationally admissible stratification of $\sM_{\rm slc}(d, \Phi_c, \Gamma, \sigma)$. Then the following statements hold:  
		\begin{itemize}  
			\item \emph{(Schematic Picard hyperbolicity)} For every quasi-finite morphism $S\to\sS$ from a reduced quasi-projective scheme $S$ and every projective completion $\overline{S}$ of $S$, the pair $(\overline{S},\overline{S}\backslash S)$ is a Picard pair.
			\item \emph{(Schematic Borel hyperbolicity)} For every quasi-finite morphism $S\to\sS$ from a reduced quasi-projective scheme $S$, $S$ is Borel hyperbolic.
			\item \emph{(Schematic Brody hyperbolicity)} For every quasi-finite morphism $S\to\sS$ from a reduced quasi-projective scheme $S$, $S$ is Brody hyperbolic.
		\end{itemize}  
	\end{thm}  
	For the moduli stack $\overline{\sM}_{g,n}$ of genus $g$ stable curves with $n$ marked points satisfying $2g - 2 + n > 0$, the boundary, defined as the simple normal crossing divisor $\partial\sM_{g,n} := \overline{\sM}_{g,n} \setminus \sM_{g,n}$, induces a canonical stratification on $\overline{\sM}_{g,n}$, with $\sM_{g,n}$ being its dense stratum. Each stratum in this stratification corresponds exactly to a distinct topological type of the fibers. This stratification is birationally admissible. The stack $\overline{\sM}_{g,n}$ admits a nice cover by a smooth projective variety that parameterizes stable curves equipped with a specific "level structure" (known as admissible $G$-covers; see \cite[Chapter XVI]{Griffiths2011}). Consequently, based on Theorem \ref{thm_main_moduli} and the existence of such a covering scheme, we derive the following hyperbolicity properties (instead of schematic hyperbolicity properties) for $\overline{\sM}_{g,n}$.
	\begin{cor}\label{cor_M_gn_stratum_hyperbolic}
		Let $\sS$ be a stratum of the canonical stratification of $\overline{\sM}_{g,n}$. Then the following statements hold:  
		\begin{itemize}  
			\item \emph{(Big Picard Theorem)} $(\overline{\sS}, \partial\sS)$ forms a Picard pair, meaning that for any holomorphic map $\gamma: \Delta^\ast \to \sS^{\rm an}$, either $0 \in \overline{\gamma^{-1}(\partial\sS)}$, or $\gamma \circ \iota_n$ can be extended to a holomorphic map $\overline{\gamma \circ \iota_n}: \Delta \to \overline{\sS}^{\rm an}$ for some $n$, where $\iota_n: \Delta^\ast \to \Delta^\ast$ is defined by $\iota_n(z) = z^n$.  
			\item \emph{(Borel Hyperbolicity)} Any holomorphic map from a finite type reduced algebraic stack to $\sS$ is algebraic.  
			\item \emph{(Brody Hyperbolicity)} For any holomorphic map $\bC \to \sS^{\rm an}$, the composition map $\bC\to S$ is constant, where $S$ is the coarse moduli space of $\sS$.
		\end{itemize}
	\end{cor}
    Motivated by the proof of Corollary \ref{cor_M_gn_stratum_hyperbolic} and a conjecture made by Javanpeykar-Sun-Zuo \cite[Conjecture 1.6]{JSZ2024}, we make the following conjecture.
    \begin{conj}\label{conj_log_uniformisable}
    	Let $\sS$ be a stratum of a birationally admissible stratification of $\sM_{\rm slc}(d, \Phi_c, \Gamma, \sigma)$. Then the pair $(\overline{\sS}, \overline{\sS} \setminus \sS)$ is logarithmically uniformisable, in the sense that there exists a surjective morphism $\eta \colon X \to \overline{\sS}$, where $X$ is a reduced projective scheme, and $\eta$ restricts to a finite \'etale morphism over $\sS$.
    \end{conj}
    The projective variety $X$ in Conjecture \ref{conj_log_uniformisable} is expected to be a fine moduli space, whose objects are those in $\overline{\sS}$ equipped with a certain "level structure". As previously noted, for every stratum $\sS$ of $\overline{\sM}_{g,n}$, the pair $(\overline{\sS}, \overline{\sS} \setminus \sS)$ is logarithmically uniformisable, where the cover is given by the fine moduli space of certain admissible $G$-covers (see Proposition \ref{prop_moduli_curve_uniformisable}). Let $\sS$ be a stratum of a birationally admissible stratification of $\sM_{\rm slc}(d, \Phi_c, \Gamma, \sigma)$. It can be shown that if the pair $(\overline{\sS}, \overline{\sS} \setminus \sS)$ is logarithmically uniformisable, then $(\overline{\sS}, \partial \sS)$ forms a Picard pair, and $\sS$ is both Borel hyperbolic and Brody hyperbolic in the stacky sense as in Corollary \ref{cor_M_gn_stratum_hyperbolic} (see also Corollary \ref{cor_log_uniform_hyperbolic}).
	\subsection{Historical accounts}
	The \emph{big Picard theorem} originates from the classical big Picard theorem, which asserts that any holomorphic map from the punctured disk $\Delta^\ast$ into $\bP^1$ omitting three distinct points can be extended to a holomorphic map $\Delta \to \bP^1$. In a recent work by Deng-Lu-Sun-Zuo \cite{DLSZ}, it was established that the big Picard theorem applies to the base of a maximal variational family of good minimal manifolds.  
	
	The \emph{Borel hyperbolicity} is a direct formal consequence of the big Picard theorem (see \cite[Theorem 3.13]{Javanpeykar2020}). 
	
	The \emph{Brody hyperbolicity}, in the context of families of canonically polarized manifolds, was established by Viehweg-Zuo \cite{VZ2003} through their construction of systems of Higgs sheaves. This result was subsequently extended by Popa-Taji-Wu \cite{PTW2019} to families of general type minimal manifolds, utilizing constructions similar to those in \cite{PS17}. These findings were further generalized by Deng \cite{Deng2022} to encompass families of good minimal manifolds. Additionally, Wei-Wu \cite{WeiWu2023} investigated the Brody hyperbolicity of log smooth families of pairs of log general type. For a more comprehensive historical overview of the hyperbolicity problem for moduli of varieties, readers are referred to \cite{DLSZ}.
	
	The study of hyperbolicity in moduli stacks is a relatively recent field. For a exploration of arithmetic hyperbolicity and Borel hyperbolicity in certain examples of moduli stacks, readers are referred to the work of Javanpeykar and Loughran \cite{Javanpeykar2021}.
    \subsection{Acknowledgement}
	The author would like to express sincere gratitude to A. Javanpeykar for his valuable comments. He kindly pointed out several errors concerning the properties of Borel hyperbolicity and Brody hyperbolicity on DM-stacks and generously shared his insights on this topic.
	\section{Hodge theoretic big Picard theorem}
	In this section, we extend the work of Deng-Lu-Sun-Zuo \cite{DLSZ} to the context of variations of mixed Hodge structures and propose a systematic framework for constructing hyperbolic spaces, contingent upon the existence of a particular admissible variation of mixed Hodge structure. These results enable us to establish hyperbolicity properties for the base spaces of families of singular varieties. The principal result of this section is presented in Theorem \ref{thm_Hodge_Picard}.
	\subsection{Variation of mixed Hodge structure}
	For the purposes of this paper, we focus on the differential-geometric properties. Consequently, in this section, we consider graded $\bR$-polarized variations of mixed Hodge structure and disregard their underlying $\bZ$-structure, even if such a structure exists.
	
	Let $S$ be a complex manifold. Denote by $\sA^k_S$ the sheaf of $C^\infty$ $k$-forms on $S$, and by $\sA^{p,q}_S$ the sheaf of $C^\infty$ $(p,q)$-forms on $S$.
	\begin{defn}  
		A \emph{pre-variation of Hodge structures} $\bV = (\bV_{\bR}, V, \nabla, F^\bullet)$ of weight $m$ on $S$ consists of the following data:  
		\begin{itemize}  
			\item An $\bR$-local system $\bV_{\bR}$ on $S$, together with a holomorphic flat connection $(V, \nabla)$ corresponding to $\bV_{\bC} := \bV_{\bR} \otimes_{\bR} \bC$ via the Riemann-Hilbert correspondence.  
			\item A regular, decreasing filtration $F^\bullet$ of holomorphic subbundles of $V$, such that each graded quotient $F^p / F^{p+1}$ is locally free, and for every point $s \in S$, the fiber $\bV(s) := (\bV_{\bR}(s), F^\bullet(s))$ forms a pure $\bR$-Hodge structure of weight $m$.  
		\end{itemize}  		
		An \emph{$\bR$-polarization} on $\bV$ is a bilinear map between $\bR$-local systems  
		$$ Q : \bV_{\bR} \otimes \bV_{\bR} \to \bR_S(-m), $$  
		which induces an $\bR$-polarization on $\bV(s)$ for every point $s \in S$.  
		
		A \emph{variation of Hodge structures} of weight $m$ is a pre-variation of Hodge structures $\bV = (\bV_{\bR}, V, \nabla, F^\bullet)$ of weight $m$ satisfying the Griffiths transversality condition:  
		$$ \nabla(F^p) \subset F^{p-1} \otimes \Omega_S, \quad \forall p. $$  
	\end{defn}
	\begin{defn}  
		A \emph{pre-variation of mixed Hodge structures} $\bV = (\bV_{\bR}, V, \nabla, W_\bullet, F^\bullet)$ on $S$ consists of the following data:  
		\begin{itemize}  
			\item An $\bR$-local system $\bV_{\bR}$ on $S$, together with a finite increasing filtration of local subsystems $W_\bullet$ of $\bV_{\bR}$.  
			\item A holomorphic flat connection $(V, \nabla)$ on $S$, along with a finite increasing filtration of sub flat connections $W_{\bullet,\bR}$, such that $(\bV_{\bR}, W_{\bullet,\bR}) \otimes_{\bR} \bC$ corresponds to $(V, \nabla, W_\bullet)$ via the Riemann-Hilbert correspondence.  
			\item A finite decreasing filtration $F^\bullet$ of holomorphic subbundles of $V$,  
		\end{itemize}  
		such that  
		$$
		{\rm Gr}^W_m(\bV) := ({\rm Gr}^W_m(\bV_{\bR}), {\rm Gr}^W_m(V), {\rm Gr}^W_m(\nabla), F^\bullet {\rm Gr}^W_m)
		$$  
		forms a pre-variation of Hodge structures of weight $m$.  
		
		A \emph{variation of mixed Hodge structures} on $S$ is a pre-variation of mixed Hodge structures $\bV = (\bV_{\bR}, V, \nabla, W_\bullet, F^\bullet)$ on $S$ satisfying the Griffiths transversality condition:  
		$$
		\nabla(F^p) \subset F^{p-1} \otimes \Omega_S, \quad \forall p.
		$$  
		
		A \emph{graded $\bR$-polarization} of a variation of mixed Hodge structures $\bV = (\bV_{\bR}, V, \nabla, W_\bullet, F^\bullet)$ consists of an $\bR$-polarization on each ${\rm Gr}^W_m(\bV)$. A variation of mixed Hodge structures is said to be \emph{graded $\bR$-polarizable} if it admits a graded $\bR$-polarization.  
	\end{defn}  
	\subsection{Admissible variation of mixed Hodge structure}
	Let $\overline{S}$ be a complex manifold, and let $E \subset \overline{S}$ be a simple normal crossing divisor. Define $S := \overline{S} \setminus E$. 
	
	Recall that for a flat connection $(V, \nabla)$ on $S$, the lower canonical extension of $(V, \nabla)$ to $\overline{S}$ is a logarithmic connection  
	\[
	\nabla: \widetilde{V} \to \widetilde{V} \otimes_{\sO_{\overline{S}}} \Omega_{\overline{S}}(\log E),
	\]  
	which extends $(V, \nabla)$ such that the real parts of the eigenvalues of the residue map along each component of $E$ lie in the interval $[0, 1)$. According to \cite[Proposition 5.4]{Deligne1970}, the lower canonical extension always exists and is unique up to isomorphism. Let $\widetilde{W}_m$ denote the lower canonical extension of $W_m$. For the notion of admissibility, we follow Kashiwara \cite{Kashiwara1986}. 
	\begin{defn}  
		Let $\bV = (\bV_{\bR}, V, \nabla, W_\bullet, F^\bullet)$ be a graded $\bR$-polarized variation of mixed Hodge structures on the punctured unit disc $\Delta^\ast$. The variation $\bV$ is called \emph{pre-admissible} with respect to $\Delta$ if the following conditions are satisfied:  
		\begin{itemize}  
			\item The monodromy operator of $\bV_{\bR}$ at the origin is quasi-unipotent.  
			\item The logarithm $N$ of the unipotent part of the residue map ${\rm Res}_0(\nabla)$ admits a weight filtration relative to $\widetilde{W}_\bullet(0)$.  
			\item The filtration $F^\bullet$ extends to a filtration $\widetilde{F}^\bullet$ of subbundles of $\widetilde{V}$ such that ${\rm Gr}_{\widetilde{F}}^p {\rm Gr}^{\widetilde{W}}_m \widetilde{V}$ is locally free for all $p$ and $m$.  
		\end{itemize}  
	\end{defn}
	\begin{defn}  
		Let $\bV = (\bV_{\bR}, V, \nabla, W_\bullet, F^\bullet)$ be a graded $\bR$-polarized variation of mixed Hodge structures on $S$. The variation $\bV$ is called \emph{admissible} with respect to $\overline{S}$ if, for every holomorphic map $f: \Delta \to \overline{S}$ such that $f(\Delta^\ast) \subset S$, the pullback $(f|_{\Delta^\ast})^\ast(\bV)$ is pre-admissible with respect to $\Delta$.  
	\end{defn}
	\begin{prop} \emph{(Kashiwara \cite[Proposition 1.11.3]{Kashiwara1986})} \label{prop_Kashiwara_extendHodge}  
		Let $\bV = (\bV_{\bR}, V, \nabla, W_\bullet, F^\bullet)$ be a graded $\bR$-polarized variation of mixed Hodge structures on $S$, which is admissible with respect to $\overline{S}$. Then the Hodge filtration $\{F^p\}$ admits an extension $\{\widetilde{F}^p\}$ on $\overline{S}$ such that the following conditions are satisfied:  
		\begin{enumerate}  
			\item $\widetilde{F}^p$ is a subbundle of $\widetilde{V}$ for all $p$.  
			\item ${\rm Gr}_{\widetilde{F}}^p {\rm Gr}_m^{\widetilde{W}}(\widetilde{V})$ is locally free for all $p$ and $m$.  
			\item $\widetilde{F}^\bullet$ satisfies the Griffiths transversality condition:  
			$$
			\nabla(\widetilde{F}^p) \subset \widetilde{F}^{p-1} \otimes \Omega_{\overline{S}}(\log E), \quad \forall p.
			$$  
		\end{enumerate}  
	\end{prop}
	\subsection{The associated logarithmic Higgs bundle}\label{section_lHB_BM}
	Let $\overline{S}$ be a complex manifold, and let $E \subset \overline{S}$ be a simple normal crossing divisor. Define $S := \overline{S} \setminus E$. Let $\bV = (\bV_{\bR}, V, \nabla, W_\bullet, F^\bullet)$ be a variation of mixed Hodge structures on $S$, which is admissible with respect to $\overline{S}$.	
	Let $(\widetilde{V}, \nabla)$ (resp. $(\widetilde{W}_\bullet, \nabla)$) denote the lower canonical extension of $(V, \nabla)$ (resp. $(W_\bullet, \nabla)$). Let $\widetilde{F}^\bullet$ be the extension of $F^\bullet$ guaranteed by Proposition \ref{prop_Kashiwara_extendHodge}, which satisfies Griffiths transversality:  
	\begin{align}\label{align_Griff_tran_Kashiwara_extension}  
		\nabla(\widetilde{F}^p) \subset \widetilde{F}^{p-1} \otimes_{\sO_{\overline{S}}} \Omega_{\overline{S}}(\log E), \quad \forall p.
	\end{align}  
	Define  
	$$
	\widetilde{H}^{p} := {\rm Gr}_{\widetilde{F}}^p(\widetilde{V}) = \widetilde{F}^p/\widetilde{F}^{p+1}, \quad \forall p.
	$$  
	Then (\ref{align_Griff_tran_Kashiwara_extension}) implies that $\nabla$ induces a $\sO_{\overline{S}}$-linear morphism  
	\begin{align}\label{align_logtheta_BM}  
		\theta: \widetilde{H}^{p} \to \widetilde{H}^{p-1} \otimes_{\sO_{\overline{S}}} \Omega_{\overline{S}}(\log E), \quad \forall p.
	\end{align}  
	Set $\widetilde{H} := \bigoplus_{p} \widetilde{H}^{p}$, and let  
	$$
	\theta: \widetilde{H} \to \widetilde{H} \otimes_{\sO_{\overline{S}}} \Omega_{\overline{S}}(\log E)
	$$  
	denote the induced map. The flatness of $\nabla$ (i.e., $\nabla^2 = 0$) ensures that $\theta^2 = 0$. Consequently, $(\widetilde{H}, \theta)$ defines a logarithmic Higgs bundle on $(\overline{S}, E)$.
	\begin{defn}\label{defn_LHB_ass_VHSBM}  
		We refer to $(\widetilde{H}, \theta)$ as the \emph{lower canonical system of logarithmic Higgs bundles} associated with the admissible variation of mixed Hodge structure $\bV$.  
	\end{defn}
	The logarithmic Higgs bundle $(\widetilde{H}, \theta)$ is equipped with a weight filtration of Higgs subbundles. For each $k$ and $p$, let  
	$$  
	\{W_m(\widetilde{H}^p) = \widetilde{W}_m({\rm Gr}_{\widetilde{F}}^p(\widetilde{V}))\}_{m \in \bZ}  
	$$  
	denote the induced weight filtration, and define  
	$$  
	W_m(\widetilde{H}) := \bigoplus_{p\leq w} W_m(\widetilde{H}^p).  
	$$  
	Then $(W_m(\widetilde{H}), \theta)$ forms a logarithmic Higgs subbundle, i.e.,  
	$$  
	\theta(W_m(\widetilde{H})) \subset W_m(\widetilde{H}) \otimes_{\sO_{\overline{S}}} \Omega_{\overline{S}}(\log E).  
	$$  
	This construction yields the sub-quotient logarithmic Higgs bundles  
	\begin{align}  
		\theta: {\rm Gr}^W_m(\widetilde{H}) \to {\rm Gr}^W_m(\widetilde{H}) \otimes_{\sO_{\overline{S}}} \Omega_{\overline{S}}(\log E).  
	\end{align}  
	These define the lower canonical system of logarithmic Higgs bundles associated with the variation of pure Hodge structure ${\rm Gr}^W_m(\bV)$. 
	\subsection{Hodge theoretic big Picard theorem}\label{section_VZ_Finsler_metric}
	Let $S$ be a smooth projective variety, and let $D \subset E \subset S$ denote two closed algebraic subsets. Let $\bV = (\bV_{\bR}, V, \nabla, W_\bullet, F^\bullet)$ represent a graded $\bR$-polarized variation of mixed Hodge structures on $S \setminus E$, which is admissible with respect to $S$. Define $w = \max\{p \mid F^p \neq 0\}$ as the maximal index associated with the Hodge filtration.
	\begin{defn}\label{defn_VMHS_maxvar}
		We say that $\bV = (\bV_{\bR}, V, \nabla, W_\bullet, F^\bullet)$ has \emph{maximal variation with respect to $(S, D)$} if the following conditions are satisfied:  
		\begin{itemize}  
			\item There exists a closed algebraic subset $Z \subset S$ such that $\mathrm{codim}_S(Z) \geq 2$, and both $D \setminus Z$ and $E \setminus Z$ are simple normal crossing divisors on $S \setminus Z$,  
			\item The lower canonical system of logarithmic Higgs bundles $(\widetilde{H} = \bigoplus_{p\leq w}\widetilde{H}^p, \theta)$ is defined on $(S \setminus Z, E \setminus Z)$ and associated with $\bV$,  
			\item An ample line bundle $L$ on $S$ and an injective morphism $L|_{S \setminus Z} \to \widetilde{H}^w$ exist,  
		\end{itemize}  
		such that the Higgs subsheaf  
		$$  
		\bigoplus_{p\geq0}L^p \subset \bigoplus_{p\geq0} \widetilde{H}^{w-p}, \quad L^p \subset \widetilde{H}^{w-p},  
		$$  
		generated by $L^0 = L|_{S \setminus Z}$, has logarithmic poles along $D$, i.e.,  
		$$  
		\theta(L^p) \subset L^{p+1} \otimes \Omega_{S \setminus Z}(\log (D \setminus Z)), \quad \forall p \geq 0.  
		$$
	\end{defn}
	The main result of this section is the following theorem.  
	\begin{thm}\label{thm_Hodge_Picard}  
		Let $S$ be a smooth projective variety, and let $D \subset E \subset S$ denote two closed algebraic subsets. Assume that there exists a graded $\bR$-polarized variation of mixed Hodge structures on $S \setminus E$, which is admissible with respect to $S$ and has maximal variation with respect to $(S, D)$. Let $\gamma: \Delta^\ast \to S \setminus D$ be a holomorphic map with Zariski-dense image. Then $\gamma$ extends to a holomorphic map $\overline{\gamma}: \Delta \to X$.  
	\end{thm}      
	The proof of Theorem \ref{thm_Hodge_Picard} will occupy the remainder of this section. The argument is inspired by the work of Deng-Lu-Sun-Zuo \cite{DLSZ}.
	\subsubsection{}
	Using the notations and assumptions introduced in Definition \ref{defn_VMHS_maxvar} and Theorem \ref{thm_Hodge_Picard}, we proceed as follows. Through a sequence of smooth blowups, we may assume that there exists a projective birational morphism $\pi: \widetilde{S} \to S$ such that $\widetilde{S}$ is a smooth projective variety and $\pi^{-1}(D)$, $\pi^{-1}(E)$ are simple normal crossing divisors on $\widetilde{S}$. With a slight abuse of notation, let $(\widetilde{H}, \theta)$ denote the lower canonical system of logarithmic Higgs bundles on $(\widetilde{S}, \pi^{-1}(E))$, associated with the pullback $\pi^{\ast}(\bV)$.     
	Then there exists an effective $\pi$-exceptional divisor $E_0$ such that the following conditions are satisfied:  
	\begin{enumerate}  
		\item The line bundle $\widetilde{L} := \pi^\ast(L) \otimes \sO_{\widetilde{S}}(-E_0)$ is ample.  
		\item There exists a natural inclusion $\widetilde{L} \subset \widetilde{H}^{w}$.  
		\item Let $(\bigoplus_{p\leq w} L^p, \theta)$ denote the logarithmic Higgs subsheaf generated by $L^0 := \widetilde{L}$, where $L^p \subset \widetilde{H}^{w-p}$. Then the logarithmic Higgs field  
		$$  
		\theta: L^p \to L^{p+1} \otimes \Omega_{\widetilde{S}}(\log \pi^{-1}(E))  
		$$  
		is holomorphic over $\widetilde{S} \setminus \pi^{-1}(D)$ for all $p \geq 0$, i.e.,  
		$$  
		\theta(L^p) \subset L^{p+1} \otimes \Omega_{\widetilde{S}}(\log \pi^{-1}(D)), \quad \forall p \geq 0.  
		$$  
	\end{enumerate}
	Let $m_0$ denote the index such that $\widetilde{L} \subset W_{m_0}(\widetilde{H}^w)$ and $\widetilde{L} \nsubseteq W_{m_0-1}(\widetilde{H}^w)$. Then there exists a natural inclusion $\widetilde{L} \subset {\rm Gr}^W_{m_0}(\widetilde{H}^w)$. Let $(\bigoplus_{p\geq 0} L^p_{m_0}, \theta)$ represent the logarithmic Higgs subsheaf of $({\rm Gr}^W_{m_0}(\widetilde{H}), \theta)$ generated by $L^0_{m_0} := \widetilde{L}$, where $L^p_{m_0} \subset {\rm Gr}^W_{m_0}(\widetilde{H}^{w-p})$. Then the logarithmic Higgs field  
	$$  
	\theta: L^p_{m_0} \to L^{p+1}_{m_0} \otimes \Omega_{\widetilde{S}}(\log \pi^{-1}(E))  
	$$  
	is holomorphic over $\widetilde{S} \setminus \pi^{-1}(D)$ for all $p \geq 0$. 
	
	Consider the following diagram:  
	\begin{align*}  
		L^0_{m_0} \stackrel{\theta}{\to} L^1_{m_0} \otimes \Omega_{\widetilde{S}}(\log \pi^{-1}(D)) \stackrel{\theta \otimes {\rm Id}}{\to} L^2_{m_0} \otimes \Omega^{\otimes 2}_{\widetilde{S}}(\log \pi^{-1}(D)) \to \cdots.  
	\end{align*}  
	This induces a map  
	$$  
	\tau_{m_0,p}: L^0_{m_0} \to {\rm Gr}^W_{m_0}(\widetilde{H}^{w-p}) \otimes \Omega^{\otimes p}_{\widetilde{S}}(\log \pi^{-1}(D))  
	$$  
	for all $p \geq 0$.  
	This further induces a map  
	\begin{align}\label{align_first_theta}  
		\rho_{m_0,p}: T^{\otimes p}_{\widetilde{S}}(-\log \pi^{-1}(D)) \to \widetilde{L}^{-1} \otimes {\rm Gr}^W_{m_0}(\widetilde{H}^{w-p})  
	\end{align}  
	for all $p \geq 0$. According to \cite[Theorem D]{Deng2022}, the map $\rho_{m_0,1}$ is generically injective. 
	\subsubsection{}
	We now define a continuous Hermitian metric on $\widetilde{L}^{-1} \otimes {\rm Gr}^W_{m_0}(\widetilde{H})$, utilizing techniques that have been developed in \cite{VZ2003,PTW2019,DLSZ}. Let $g_{\widetilde{L}}$ denote a smooth Hermitian metric on $\sO_{\widetilde{S}}(\widetilde{L})$ with positive curvature. Let $\pi^{-1}(E) = \sum E_i$ represent the irreducible decomposition, and let $g_{i}$ denote a Hermitian metric on $\sO_{\widetilde{S}}(E_i)$ for each $i$. Let $s_i \in \Gamma(\widetilde{S}, \sO_{\widetilde{S}}(E_i))$ represent the defining section of $E_i$, and define $g_{\alpha} := (\prod_i \log |s_i|^2_{g_i})^{\alpha} g_{\widetilde{L}}$. Let $h_{m_0}$ denote a Hodge metric on ${\rm Gr}^W_{m_0}(H)$, regarded as a singular Hermitian metric on ${\rm Gr}^W_{m_0}(\widetilde{H})$. 
	\begin{prop}\label{prop_MHmetric_upperbound}  
		There exists an integer $N$ such that the following statements are satisfied. Let $x$ be a point on $E$, and let $(U; z_1, \dots, z_n)$ denote holomorphic local coordinates on an open neighborhood $U$ of $x = (0, \dots, 0)$ in $X$, with the property that $E = \{z_1 \cdots z_r = 0\}$. Let $s \in {\rm Gr}^W_{m_0}(\widetilde{H})(U)$ be a holomorphic section. Then,  
		\begin{align}\label{align_norm_upperbound}  
			|s|_{h_{m_0}} \lesssim |z_1 \cdots z_r|^{-N}  
		\end{align}  
		near $x$.
	\end{prop}
	\begin{proof}    	
		It suffices to establish inequality (\ref{align_norm_upperbound}) for any holomorphic section $s \in {\rm Gr}^{\widetilde{W}}_{m_0}(\widetilde{V})(U)$. Let $\bV_\infty$ denote the space of germs of multivalued flat sections of $({\rm Gr}^W_{m_0}(V), \nabla)$ at $x$. Define $E_i := \{z_i = 0\}$ for $i = 1, \dots, r$. Let $T_1, \dots, T_r$ represent the monodromy operators of ${\rm Gr}^W_{m_0}(\bV_{\bC})$ around $E_1, \dots, E_r$, respectively. Since $T_1, \dots, T_r$ are pairwise commutative, there exists a finite decomposition  
		$$
		\bV_\infty = \bigoplus_{0 \leq \alpha_1, \dots, \alpha_r < 1} \bV_{\alpha_1, \dots, \alpha_r}
		$$  
		such that $(T_i - e^{2\pi\sqrt{-1}\alpha_i}{\rm Id})$ is unipotent on $\bV_{\alpha_1, \dots, \alpha_r}$ for each $i = 1, \dots, r$.  
		
		Let  
		$$
		v_1, \dots, v_k \in \bigcup_{0 \leq \alpha_1, \dots, \alpha_r < 1} \bV_{\alpha_1, \dots, \alpha_r}
		$$  
		form a basis of $\bV_\infty$. Then $\widetilde{v_1}, \dots, \widetilde{v_k}$, determined by  
		\begin{align}\label{align_adapted_frame}  
			\widetilde{v_j} := {\rm exp}\left(\sum_{i=1}^r \log z_i (\alpha_i{\rm Id} + N_i)\right)v_j \quad \text{if } v_j \in \bV_{\alpha_1, \dots, \alpha_r}, \quad \forall j = 1, \dots, k,
		\end{align}  
		be a local frame of $\widetilde{V}$ at $x$. According to \cite[Theorem 5.21]{Cattani_Kaplan_Schmid1986}, one has  
		$$
		|\widetilde{v_j}|_{h_{m_0}} \lesssim |z_1 \cdots z_r|^{-N}, \quad \forall j = 1, \dots, k,
		$$  
		when $N$ is sufficiently large. Since $\overline{S}$ is compact, there exists a uniform constant $N$ such that (\ref{align_norm_upperbound}) holds. This concludes the proof.
	\end{proof}
	According to Proposition \ref{prop_MHmetric_upperbound}, there exists an integer $\alpha_0$ such that for every $\alpha \geq \alpha_0$, the metric $g_{\alpha}^{-1} h_{m_0}$ extends to a continuous Hermitian metric (still denoted by $g_{\alpha}^{-1} h_{m_0}$) on $\widetilde{L}^{-1} \otimes {\rm Gr}^W_{m_0}(\widetilde{H})$, with the property that $g_{\alpha}^{-1} h_{m_0}$ degenerates along $\pi^{-1}(E)$, i.e., $g_{\alpha}^{-1} h_{m_0}(s) \equiv 0$ for every point $s \in \pi^{-1}(E)$.
	
	After appropriately rescaling $h_i$, we may assume that  
	\begin{align}\label{align_curvature_halpha}  
		\Theta_{g_{\alpha}}(\sO_{\widetilde{S}}(\widetilde{L})) \geq \omega,  
	\end{align}  
	where $\omega$ denotes some K\"ahler form on $\widetilde{S}$ (\cite[Lemma 2.16]{DLSZ}). In the remainder of this section, we fix $g_i$ and $\alpha > 0$ such that the Hermitian metric $g_{\alpha}^{-1} h_{m_0}$ on $\widetilde{L}^{-1} \otimes {\rm Gr}^W_{m_0}(\widetilde{H})$ degenerates along $\pi^{-1}(E)$, i.e., $g_{\alpha}^{-1} h_{m_0}|_{\pi^{-1}(E)} \equiv 0$, and such that (\ref{align_curvature_halpha}) holds.
	\subsubsection{}
	\begin{defn}[Finsler Metric]  
		Let $E$ be a holomorphic vector bundle over a complex manifold $X$. A Finsler pseudometric on $E$ is a continuous function $h: E \to [0, \infty)$ satisfying the following property:  
		$$  
		h(av) = |a| h(v), \quad \forall a \in \bC, \quad \forall v \in E.  
		$$   
	\end{defn} 
	Let $U_0 \subset \widetilde{S} \setminus \pi^{-1}(D)$ denote a Zariski-dense open subset such that $\rho_{m_0,1}$ is injective on $U_0$. Let $\gamma': \Delta^\ast \to \widetilde{S} \setminus \pi^{-1}(D)$ be the holomorphic map lifting $\gamma$. It suffices to show that $\gamma'$ extends to a holomorphic morphism $\Delta \to \widetilde{S}$. Since ${\rm Im}(\gamma')$ is Zariski-dense in $\widetilde{S}$ and $\rho_{m_0,1}$ is injective on $U_0$, the natural map  
	$$  
	\tau_{\gamma,1}: \gamma^\ast \widetilde{L} \stackrel{\gamma^\ast \tau_1}{\to} \gamma^\ast {\rm Gr}^W_{m_0}(\widetilde{H}^{w-1}) \otimes \gamma^\ast \Omega_{\widetilde{S}}(\log \pi^{-1}(D)) \stackrel{{\rm Id} \otimes d\gamma}{\to} \gamma^\ast {\rm Gr}^W_{m_0}(\widetilde{H}^{w-1}) \otimes \Omega_{\Delta^\ast}  
	$$  
	is nonzero. Hence, there exists a minimal integer $n_\gamma \geq 1$ such that  
	$$  
	\tau_{\gamma,n_\gamma}: \gamma^\ast \widetilde{L} \stackrel{\gamma^\ast \tau_{n_\gamma}}{\to} \gamma^\ast {\rm Gr}^W_{m_0}(\widetilde{H}^{w-n_\gamma}) \otimes \gamma^\ast \Omega^{\otimes n_\gamma}_{\widetilde{S}}(\log \pi^{-1}(D)) \stackrel{{\rm Id} \otimes d\gamma}{\to} \gamma^\ast {\rm Gr}^W_{m_0}(\widetilde{H}^{w-n_\gamma}) \otimes \Omega^{\otimes n_\gamma}_{\Delta^\ast}  
	$$  
	is nonzero, while the composition  
	$$  
	\gamma^\ast \widetilde{L} \stackrel{\tau_{\gamma,n_\gamma}}{\to} \gamma^\ast {\rm Gr}^W_{m_0}(\widetilde{H}^{w-n_\gamma}) \otimes \Omega^{\otimes n_\gamma}_{\Delta^\ast} \to \gamma^\ast {\rm Gr}^W_{m_0}(\widetilde{H}^{w-n_\gamma-1}) \otimes \Omega^{\otimes (n_\gamma+1)}_{\Delta^\ast}  
	$$  
	vanishes. This yields a nonzero map  
	$$  
	\widetilde{L} \to {\rm Gr}^W_{m_0}(\widetilde{H}^{w-n_\gamma}) \otimes \Omega^{\otimes n_\gamma}_{\widetilde{S}}(\log \pi^{-1}(D)).  
	$$  
	This induces a map  
	\begin{align}  
		\iota_{n_\gamma}: T^{\otimes n_\gamma}_{\widetilde{S}}(-\log \pi^{-1}(D)) \to \widetilde{L}^{-1} \otimes {\rm Gr}^W_{m_0}(\widetilde{H}^{n_\gamma}).  
	\end{align}  
	The pullback metric $\iota_{n_\gamma}^\ast(g_\alpha^{-1}h_{m_0})$ defines a Finsler pseudometric $h_\gamma$ on $T_{\widetilde{S}}(-\log \pi^{-1}(D))$ via  
	$$  
	|v|_{h_\gamma} := |\iota_{n_\gamma}(v^{\otimes n_\gamma})|^{\frac{1}{n_\gamma}}_{g_\alpha^{-1}h_{m_0}}, \quad v \in T_{\widetilde{S}}(-\log \pi^{-1}(D)).  
	$$ 
	\begin{prop}\label{prop_negative_curvature}  
		With the same notations as above, $|\frac{\partial}{\partial z}|^2_{\gamma^\ast h_\gamma}$ is not identically zero, and the following inequality holds in the sense of currents:  
		$$  
		\partial\overline{\partial} \log \left| \frac{\partial}{\partial z} \right|^2_{\gamma^\ast h_\gamma} \geq \frac{1}{n} \gamma^\ast (\Theta_{g_\alpha}(L)).  
		$$  
	\end{prop}  	
	\begin{proof}  
		Observe that $\gamma^\ast h_\gamma$ defines a (possibly degenerate) generically smooth continuous Hermitian metric on $T_{\Delta^\ast}$. The first claim follows from the fact that $\tau_{\gamma,n_\gamma}$ is nonzero.  
		
		By the Poincar\'e-Lelong equation, we have  
		$$  
		\partial\overline{\partial} \log \left| \frac{\partial}{\partial z} \right|^2_{\gamma^\ast h_\gamma} = -\Theta_{\gamma^\ast h_\gamma}(T_{\Delta^\ast}) + R,  
		$$  
		where $R$ denotes the ramification divisor of $\gamma$. Let $N$ denote the saturation of the image of $d\gamma: T_{\Delta^\ast} \to \gamma^\ast T_{\widetilde{S}}(-\log \pi^{-1}(D))$. Then one has  
		\begin{align*}  
			\Theta_{\gamma^\ast h_\gamma}(T_{\Delta^\ast}) &\leq \Theta_{\gamma^\ast h_\gamma}(N) = \frac{1}{n_\gamma} \Theta_{\gamma^\ast h_\gamma^{n_\gamma}}(N^{\otimes n_\gamma}) \\  
			&\leq \frac{1}{n_\gamma} \gamma^\ast \Theta_{h_\gamma^{n_\gamma}}\left(T^{\otimes n_\gamma}_{\widetilde{S}}(-\log \pi^{-1}(D))\right)\big|_{N^{\otimes n_\gamma}} \\  
			&\leq \frac{1}{n_\gamma} \gamma^\ast \Theta_{g_\alpha^{-1}h_{m_0}}\left(\widetilde{L}^{-1} \otimes {\rm Gr}^W_{m_0}(\widetilde{H}^{n_\gamma})\right)\big|_{\gamma^\ast(\iota_{n_\gamma}(N^{\otimes n_\gamma}))} \\  
			&= -\frac{1}{n_\gamma} \gamma^\ast \Theta_{g_\alpha}(\widetilde{L}) + \frac{1}{n_\gamma} \gamma^\ast \Theta_{h_{m_0}}\left({\rm Gr}^W_{m_0}(\widetilde{H}^{n_\gamma})\right)\big|_{\gamma^\ast(\iota_{n_\gamma}(N^{\otimes n_\gamma}))},  
		\end{align*}  
		as $(1,1)$-forms on $\Delta^\ast$.  
		
		By the definition of $n_\gamma$, it follows that $\gamma^\ast(\iota_{n_\gamma}(N^{\otimes n_\gamma}))$ lies in the kernel of the Higgs field  
		$$  
		\theta_\gamma: \gamma^\ast {\rm Gr}^W_{m_0}(\widetilde{H}^{n_\gamma}) \to \gamma^\ast {\rm Gr}^W_{m_0}(\widetilde{H}^{n_\gamma+1}) \otimes \Omega_{\Delta^\ast}  
		$$  
		of the Higgs sheaf $(\gamma^\ast{\rm Gr}^W_{m_0}(H), \theta_\gamma := \gamma^\ast(\theta))$.  		
		Using Griffiths' curvature formula  
		$$  
		\gamma^\ast \Theta_{h_{m_0}}({\rm Gr}^W_{m_0}(H)) + \theta_\gamma \wedge \overline{\theta_\gamma} + \overline{\theta_\gamma} \wedge \theta_\gamma = 0,  
		$$  
		we deduce that  
		$$  
		\gamma^\ast \Theta_{h_{m_0}}\left({\rm Gr}^W_{m_0}(H^{n_\gamma})\right)\big|_{\gamma^\ast(\iota_{n_\gamma}(N^{\otimes n_\gamma}))} = -\theta_\gamma \wedge \overline{\theta_\gamma}\big|_{\gamma^\ast(\iota_{n_\gamma}(N^{\otimes n_\gamma}))} \leq 0.  
		$$  		
		Combining the above formulas, the proposition is proved.  
	\end{proof}
	\subsubsection{}
	To establish the main result of this section, we employ the following criterion by Deng-Lu-Sun-Zuo \cite{DLSZ}, which relates the big Picard-type theorem to the negativity of the complex sectional curvature of certain Finsler pseudometrics.    
	\begin{thm}{\cite[Theorem A]{DLSZ}}\label{thm_big_Picard_LSZ}  
		Let $(X, \omega)$ be a smooth projective variety equipped with a K\"ahler form $\omega$, and let $D$ be a simple normal crossing divisor on $X$. Let $\gamma: \Delta^\ast \to X \setminus D$ denote a holomorphic map. Assume that there exists a Finsler pseudometric $h$ on $T_X(-\log D)$ such that the following conditions are satisfied:  
		\begin{enumerate}  
			\item $\gamma^\ast h$ defines a Finsler pseudometric that is not identically zero.  
			\item The inequality  
			$$  
			\partial\overline{\partial} \log \left| \frac{\partial}{\partial z} \right|^2_{\gamma^\ast h} \geq \gamma^\ast(\omega)  
			$$  
			holds in the sense of currents.  
		\end{enumerate}  
		Then $\gamma$ extends to a holomorphic map $\overline{\gamma}: \Delta \to X$.  
	\end{thm}
	By (\ref{align_curvature_halpha}), Proposition \ref{prop_negative_curvature}, and Theorem \ref{thm_big_Picard_LSZ}, we conclude that the holomorphic map $\gamma': \Delta^\ast \to \widetilde{S} \setminus \pi^{-1}(D)$ extends to a holomorphic map $\overline{\gamma'}: \Delta \to \widetilde{S}$. Consequently, the composition $\pi \circ \overline{\gamma'}: \Delta \to S$ defines a holomorphic map that extends $\gamma$. This completes the proof of Theorem \ref{thm_Hodge_Picard}.
	\section{Admissible families of stable minimal models}\label{section_boundedness}	
	\subsection{Stable minimal models and their moduli}\label{section_moduli}
	We now review the main results from \cite{Birkar2022} that will be utilized in the subsequent sections. For the purposes of this article, all schemes are defined over ${\rm Spec}(\bC)$. A \emph{stable minimal model} is a triple $(X, B), A$, where $X$ is a reduced, connected, projective scheme of finite type over ${\rm Spec}(\bC)$, and $A, B \geq 0$ are $\bQ$-divisors satisfying the following conditions:  
	\begin{itemize}  
		\item $(X, B)$ is a projective, connected slc (semi-log-canonical) pair,  
		\item $K_X + B$ is semi-ample,  
		\item $K_X + B + tA$ is ample for some $t > 0$, and  
		\item $(X, B + tA)$ is slc for some $t > 0$.  
	\end{itemize}  
	Let  
	$$
	d \in \bN, \, c \in \bQ^{\geq 0}, \, \Gamma \subset \bQ^{>0} \text{ a finite set, and } \sigma \in \bQ[t].
	$$  
	A $(d, \Phi_c, \Gamma, \sigma)$-stable minimal model is a stable minimal model $(X, B), A$ satisfying the following conditions:  
	\begin{itemize}  
		\item $\dim X = d$,  
		\item the coefficients of $A$ and $B$ belong to $c \bZ^{\geq 0}$,  
		\item ${\rm vol}(A|_F) \in \Gamma$, where $F$ is any general fiber of the fibration $f: X \to Z$ determined by $K_X + B$, and  
		\item ${\rm vol}(K_X + B + tA) = \sigma(t)$ for $0 \leq t \ll 1$.  
	\end{itemize}  
	Let $S$ be a reduced scheme over ${\rm Spec}(\bC)$. A family of $(d, \Phi_c, \Gamma, \sigma)$-stable minimal models over $S$ consists of a projective morphism $X \to S$ of schemes and $\bQ$-divisors $A$ and $B$ on $X$, satisfying the following conditions:  
	\begin{itemize}  
		\item $(X, B + tA) \to S$ is a locally stable family (i.e., $K_{X/S} + B + tA$ is $\bQ$-Cartier) for every sufficiently small rational number $t \geq 0$,  
		\item $A = cN$, $B = cD$, where $N, D \geq 0$ are relative Mumford divisors, and  
		\item $(X_s, B_s), A_s$ is a $(d, \Phi_c, \Gamma, \sigma)$-stable minimal model for each point $s \in S$.  
	\end{itemize} 
	Let ${\rm Sch}_{\bC}^{\rm red}$ denote the category of reduced schemes defined over ${\rm Spec}(\bC)$. Define  
	$$
	\sM^{\rm red}_{\rm slc}(d, \Phi_c, \Gamma, \sigma): S \mapsto \{\text{families of } (d, \Phi_c, \Gamma, \sigma)\text{-stable minimal models over } S\},
	$$  
	a functor of groupoids over ${\rm Sch}_{\bC}^{\rm red}$.  
	
	\begin{thm}[Birkar \cite{Birkar2022}] \label{thm_moduli_stable_var}  
		There exists a proper Deligne-Mumford stack $\sM_{\rm slc}(d, \Phi_c, \Gamma, \sigma)$ over ${\rm Spec}(\bC)$ such that the following properties hold:  
		\begin{itemize}  
			\item $\sM_{\rm slc}(d, \Phi_c, \Gamma, \sigma)|_{{\rm Sch}_{\bC}^{\rm red}} = \sM^{\rm red}_{\rm slc}(d, \Phi_c, \Gamma, \sigma)$ as functors of groupoids.  
			\item $\sM_{\rm slc}(d, \Phi_c, \Gamma, \sigma)$ admits a projective good coarse moduli space $M_{\rm slc}(d, \Phi_c, \Gamma, \sigma)$.  
		\end{itemize}  
	\end{thm}  
	\begin{proof}  
		See the proof of \cite[Theorem 1.14]{Birkar2022}. Using the notations in \cite[\S 10.7]{Birkar2022}, we have  
		$$
		\sM_{\rm slc}(d, \Phi_c, \Gamma, \sigma) = \left[M_{\rm slc}^e(d, \Phi_c, \Gamma, \sigma, a, r, \bP^n)/{\rm PGL}_{n+1}(\bC)\right],
		$$  
		where the right-hand side denotes the stacky quotient.  
	\end{proof}
	\subsection{Polarization on $M_{\rm slc}(d,\Phi_c,\Gamma,\sigma)$}\label{section_polarization_moduli}
	In this section, we consider certain natural ample $\bQ$-line bundles on $M_{\rm slc}(d, \Phi_c, \Gamma, \sigma)$. Their constructions are implicitly described in the proof of \cite[Theorem 1.14]{Birkar2022}, relying on Koll\'ar's ampleness criterion \cite{Kollar1990}.  
	
	Fix the data $d, \Phi_c, \Gamma, \sigma$. Since $\sM_{\rm slc}(d, \Phi_c, \Gamma, \sigma)$ is of finite type, there exist constants  
	$$
	(a, r, j) \in \bQ^{\geq 0} \times (\bZ^{>0})^2,
	$$  
	depending only on $d, \Phi_c, \Gamma, \sigma$, such that every $(d, \Phi_c, \Gamma, \sigma)$-stable minimal model $(X, B), A$ satisfies the following conditions (cf. \cite[Lemma 10.2]{Birkar2022}):  
	\begin{itemize}  
		\item $(X, B + aA)$ is an slc pair,  
		\item $r(K_X + B + aA)$ is a very ample integral Cartier divisor with  
		$$
		H^i(X, \sO_X(k r(K_X + B + aA))) = 0, \quad \forall i > 0, \forall k > 0,  
		$$  
		\item the embedding $X \hookrightarrow \bP(H^0(X, r(K_X + B + aA)))$ is defined by equations of degree $\leq j$, and  
		\item the multiplication map  
		$$
		S^j(H^0(X, \sO_X(r(K_X + B + aA)))) \to H^0(X, \sO_X(j r(K_X + B + aA)))
		$$  
		is surjective.  
	\end{itemize}
	\begin{defn}
		A tuple $(a, r, j) \in \bQ^{\geq 0} \times (\bZ^{>0})^2$ that satisfies the conditions above is referred to as a \emph{$(d, \Phi_c, \Gamma, \sigma)$-polarization datum}.
	\end{defn}
	Let $(a, r, j) \in \bQ^{\geq 0} \times (\bZ^{>0})^2$ be a $(d, \Phi_c, \Gamma, \sigma)$-polarization datum. Let $(X, B), A \to S$ be a family of $(d, \Phi_c, \Gamma, \sigma)$-stable minimal models. Then $f_\ast(r(K_{X/S} + B + aA))$ is locally free and commutes with arbitrary base changes. Therefore, the assignment  
	$$
	f: (X, B), A \to S \in \sM_{\rm slc}(d, \Phi_c, \Gamma, \sigma)(S) \mapsto f_\ast(r(K_{X/S} + B + aA))
	$$  
	defines a locally free coherent sheaf on the stack $\sM_{\rm slc}(d, \Phi_c, \Gamma, \sigma)$, denoted by $\Lambda_{a,r}$. Let $\lambda_{a,r} := \det(\Lambda_{a,r})$. Since $\sM_{\rm slc}(d, \Phi_c, \Gamma, \sigma)$ is Deligne-Mumford, some power $\lambda_{a,r}^{\otimes k}$ descends to a line bundle on $M_{\rm slc}(d, \Phi_c, \Gamma, \sigma)$. For this reason, we regard $\lambda_{a,r}$ as a $\bQ$-line bundle on $M_{\rm slc}(d, \Phi_c, \Gamma, \sigma)$.
	\begin{prop}\label{prop_ample_line_bundle_moduli}  
		Let $(a, r, j) \in \bQ^{\geq 0} \times (\bZ^{>0})^2$ be a $(d, \Phi_c, \Gamma, \sigma)$-polarization datum. Then $\lambda_{a,r}$ is ample on $M_{\rm slc}(d, \Phi_c, \Gamma, \sigma)$.  
	\end{prop}  	
	\begin{proof}  
		By the same arguments as in \cite[\S 2.9]{Kollar1990}. It suffices to show that $f_\ast(r(K_{X/S} + B + aA))$ is nef when $S$ is a smooth projective curve. This was established by Fujino \cite{Fujino2018} and Kov\'acs-Patakfalvi \cite{Kovacs2017}.  
	\end{proof}  
	\subsection{Higgs sheaves associated to an admissible stable family}\label{section_VZHiggs_stable_family}
	The objective of this section is to define the notion of an admissible family of stable minimal models and to review the Viehweg-Zuo Higgs sheaf associated with such families (Theorem \ref{thm_big_Higgs_sheaf}). Intuitively, the admissibility condition implies that the family admits an equisingular birational model.
	\begin{defn}\label{defn_semipair}
		A \emph{semi-pair} $(X, \Delta)$ consists of a reduced scheme of finite type over ${\rm Spec}(\bC)$, of pure dimension, together with a $\bQ$-divisor $\Delta\geq 0$ on $X$, satisfying the following conditions:  
		\begin{enumerate}  
			\item $X$ is an $S_2$-scheme with nodal singularities in codimension one,  
			\item no component of ${\rm Supp}(\Delta)$ is contained in the singular locus of $X$,  
			\item $K_X + \Delta$ is $\bQ$-Cartier.  
		\end{enumerate}  
		A \emph{projective semi-pair} is a semi-pair $(X, \Delta)$ where $X$ is a projective scheme.
	\end{defn}
	\begin{defn}
		Let $(X, \Delta)$ be a projective semi-pair. A morphism between projective semi-pairs $\pi: (X', \Delta') \to (X, \Delta)$ is said to be a \textit{semi-log resolution} if $\pi: X' \to X$ is a semi-log resolution of singularities (\cite{Bierstone2013} or \cite[\S 10.4]{Kollar2013}) such that there exists a $\pi$-exceptional $\bQ$-divisor $E \geq 0$ on $X'$ satisfying the relation $K_{X'} + \Delta' = \pi^{\ast}(K_X + \Delta) + E$.
	\end{defn}
	\begin{defn}[Simple normal crossing family]
		Let $S$ and $X$ be reduced schemes of finite type over $\operatorname{Spec}(\mathbb{C})$, and let $D \geq 0$ be a Weil $\bQ$-divisor on $X$. Let $f: X \to S$ be a morphism. The morphism $f: (X, D) \to S$ is said to be \emph{simple normal crossing over $S$} at a point $x \in X$ if there exists a Zariski open neighborhood $U$ of $x$ in $X$ that can be embedded into a scheme $Y$, which is smooth over $S$. In this embedding, $Y$ admits a regular system of parameters $(z_1, \dots, z_p, y_1, \dots, y_r)$ over $S$ at the point corresponding to $x = 0$, such that $U$ is defined by the monomial equation $z_1 \cdots z_p = 0$ and  
		\[ D|_U = \sum_{i=1}^r a_i (y_i = 0)|_U, \quad \text{where } a_i \geq 0, \]  
		over $S$.
		
		The pair $f: (X, D) \to S$ is said to be a \emph{simple normal crossing family over $S$} if it satisfies the simple normal crossing condition over $S$ at every point of $X$. A family is referred to as a \emph{log smooth family over $S$} if it is a simple normal crossing family over $S$ and each fiber $X_s := f^{-1}(s)$, for $s \in S$, is smooth (not necessarily connected).
		
		In the special case where $S = \operatorname{Spec}(\mathbb{C})$, the pair $(X, D)$ is referred to as a \emph{simple normal crossing pair} if $(X, D)$ forms a simple normal crossing family over $\operatorname{Spec}(\mathbb{C})$. In this context, $X$ has Gorenstein singularities and possesses an invertible dualizing sheaf $\omega_X$. The canonical divisor $K_X$ is defined up to linear equivalence via the isomorphism $\omega_X \simeq \sO_X(K_X)$.
		
		A simple normal crossing pair $(X, D)$ is called a \emph{log smooth pair} if $X$ is smooth (note that $X$ is not necessarily connected). In this case, the support $\operatorname{Supp}(D)$ has simple normal crossings.
	\end{defn}
	\begin{defn}
		Let $(X_1, \Delta_1)$ and $(X_2, \Delta_2)$ be two projective semi-pairs over a variety $S$. A rational map $f: (X_1, \Delta_1) \dasharrow (X_2, \Delta_2)$ over $S$ is said to be a \textit{log birational map over $S$} if there exists a common semi-log resolution $\pi_i: (X', \Delta') \to (X_i, \Delta_i)$, $i = 1, 2$, over $S$ such that $f = \pi_2 \circ \pi_1^{-1}$. We say that a morphism $f: (X, \Delta) \to S$ admits a \textit{simple normal crossing log birational model} if there exists a simple normal crossing family $f': (X', \Delta') \to S$ and a log birational map $\alpha:(X, \Delta) \dasharrow (X', \Delta')$ over $S$. If moreover $\alpha_s:(X_s,\Delta_s)\dasharrow (X'_s,\Delta'_s)$ is a log birational map for every $s\in S$, we say that $f: (X, \Delta) \to S$ admits a \textit{simple normal crossing strictly log birational model}.
	\end{defn}
	\begin{defn}\label{defn_admissible}
		Let $f: (X, B), A \to S$ be a family of stable minimal models over a variety $S$. The morphism $f$ is said to be \emph{birationally admissible} (resp. \emph{strictly birationally admissible}) if $(X, B + A) \to S$ admits a log smooth log birational model (resp. log smooth strictly log birational model).
	\end{defn}
	A typical example of a strictly admissible family of stable minimal models is a family $f: (X, B), A \to S$ such that $(X, B + A) \to S$ admits a simultaneous resolution. More precisely, there exists a semi-log resolution $(X', \Delta') \to (X, B + A)$ satisfying the following conditions: $(X', \Delta') \to S$ is a simple normal crossing family, and for each $s \in S$, the induced morphism $(X'_s, \Delta'_s) \to (X_s, B_s + A_s)$ is a semi-log resolution.
	\begin{thm}\label{thm_big_Higgs_sheaf}
		Let $f^o: (X^o, B^o), A^o \to S^o$ be a birationally admissible family of $(d, \Phi_c, \Gamma, \sigma)$-stable minimal models over a smooth quasi-projective variety $S^o$, defining a generically finite morphism $\xi^o: S^o \to \sM_{\rm slc}(d, \Phi_c, \Gamma, \sigma)$. Let $S$ be a smooth projective variety containing $S^o$ as a Zariski open subset such that $D := S \setminus S^o$ is a reduced simple normal crossing divisor, and $\xi^o$ extends to a morphism $\xi: S \to M_{\rm slc}(d, \Phi_c, \Gamma, \sigma)$. Let $\sL$ be a line bundle on $S$. Then there exist the following data:  
		\begin{enumerate}
			\item An algebraic subset $Z \subset S$ of codimension $\geq 2$ such that $(U := S \setminus Z, U \cap D)$ is a log smooth pair, and a simple normal crossing divisor $E \subset U$ containing $U \cap D$.
			\item A graded $\bQ$-polarized admissible variation of mixed Hodge structure $\bV$ on $U \setminus E$, which in turn gives rise to the lower canonical system of logarithmic Higgs bundles $(\widetilde{H} = \bigoplus_{p\leq w} \widetilde{H}^{p}, \theta)$ on $(U,E)$ (Definition \ref{defn_LHB_ass_VHSBM}).
		\end{enumerate}
		These data satisfy the following conditions:  
		\begin{enumerate}
			\item There exists a natural inclusion $\sL|_U \subset \widetilde{H}^{w}$.  
			\item Let $(\bigoplus_{p\geq0} L^p, \theta)$ be the logarithmic Higgs subsheaf generated by $L^0 := \sL|_U$, where $L^p \subset \widetilde{H}^{w-p}$. Then the logarithmic Higgs field  
			$$
			\theta: L^p \to L^{p+1} \otimes \Omega_U(\log E)
			$$  
			is holomorphic over $U\setminus D$ for each $p\geq 0$, i.e.,  
			$$
			\theta(L^p) \subset L^{p+1} \otimes \Omega_U(\log D \cap U),\quad p\geq 0.
			$$  
		\end{enumerate}
	\end{thm}
	\subsection{Hyperbolicity properties for admissible families of stable minimal models}
	\begin{defn}
		Let $X$ be a projective variety, and let $Z \subset X$ be a closed algebraic subset. The pair $(X, Z)$ is called a \emph{Picard pair} if, for any holomorphic map $\gamma: \Delta^\ast \to X$ from the punctured unit disc $\Delta^\ast$, either $0 \in \overline{\gamma^{-1}(Z)}$, or $\gamma$ can be extended to a holomorphic map $\overline{\gamma}: \Delta \to X$. A quasi-projective $Y$ is called \emph{Picard hyperbolic} if for any projective variety $\overline{Y}$ containing $Y$ as a dense Zariski open subset, $(\overline{Y},\overline{Y}\backslash Y)$ is a Picard pair.
	\end{defn}
	The classical Big Picard theorem is equivalent to the statement that $\bP^1\backslash\{0, 1, \infty\}$ is a Picard hyperbolic.
	\begin{defn}
		A complex space $X$ is called \emph{Brody hyperbolic} if there is no non-constant holomorphic map $\bC\to X$.
	\end{defn}
	\begin{defn}[Javanpeykar-Kucharczyk \cite{Javanpeykar2020}]
		An algebraic variety $X$ is called \emph{Borel hyperbolic} if any holomorphic map from a finite type reduced scheme to $X$ is algebraic.
	\end{defn}
	\begin{prop}\label{prop_BPT_BH}
		Let $(X,Z)$ be a Picard pair. Then $X\backslash Z$ is Borel hyperbolic and Brody hyperbolic.
	\end{prop}
	\begin{proof}
		See \cite[Theorem 3.13]{Javanpeykar2020}.
	\end{proof}
	Readers are referred to \cite{Lang1987,   Javanpeykar2020} for further details on the relationships between various hyperbolicity conditions.
	\begin{thm}\label{thm_main_proof}
		Let $f^o:(X^o,B^o),A^o\to S^o$ be a strictly admissible family of $(d,\Phi_c,\Gamma,\sigma)$ stable minimal models over a quasi-projective variety $S^o$ which defines a quasi-finite morphism $\xi^o:S^o\to M_{\rm slc}(d,\Phi_c,\Gamma,\sigma)$.
		$S^o$ is Picard hyperbolic, Borel hyperbolic and Brody hyperbolic.
	\end{thm}
	\begin{proof}
		Let $S$ be a projective variety containing $S^o$ as a Zariski open subset. 
		Let $\gamma: \Delta^\ast \to S$ be a holomorphic map such that $0 \notin \overline{\gamma^{-1}(S \setminus S^o)}$. We aim to demonstrate that $\gamma$ extends holomorphically to $0 \in \Delta$. By shrinking $\Delta$, we may assume that ${\rm Im}(\gamma) \subset S^o$.  
		
		Take $B \subset S$ as the Zariski closure of ${\rm Im}(\gamma)$. Let $\pi: B' \to B$ be a desingularization such that $\pi$ is biholomorphic over $B_{\rm reg} \cap S^o$, $\pi^{-1}(B \setminus (B_{\rm reg} \cap S^o))$ is a simple normal crossing divisor on $B'$ and classifying map can be extended to a morphism $S \to M_{\text{slc}}(d, \Phi_c, \Gamma, \sigma)$. Since $\pi$ is a proper map, $\gamma$ can be lifted to $\gamma': \Delta^\ast \to \pi^{-1}(B \cap S^o)$. It suffices to show that $\gamma'$ extends to a holomorphic map $\overline{\gamma'}: \Delta \to B'$. By taking the base change of $f^o$ via $\pi^{-1}(B \cap S^o) \to S^o$, we may proceed under the following assumptions without loss of generality.
		\begin{framed}
			Let $S$ be a smooth projective variety, and let $D := S \setminus S^o$ be a simple normal crossing divisor on $S$. Suppose $\gamma: \Delta^\ast \to S \setminus D$ is a holomorphic map such that $\text{Im}(\gamma)$ is Zariski dense in $S$. Additionally, there exists an admissible family of $(d, \Phi_c, \Gamma, \sigma)$-stable minimal models over $S^o$, whose classifying map is generically finite and can be extended to a morphism $S \to M_{\text{slc}}(d, \Phi_c, \Gamma, \sigma)$.
		\end{framed}
		By selecting $L$ in Theorem \ref{thm_big_Higgs_sheaf} as an ample line bundle on $S$, there exists a closed algebraic subset $E \supseteq D$, together with a graded $\bR$-polarized variation of mixed Hodge structures on $S \setminus E$. This variation is admissible with respect to $S$ and exhibits maximal variation with respect to the pair $(S, D)$.  
		
		In light of Theorem \ref{thm_Hodge_Picard}, the map $\gamma$ extends holomorphically to $\overline{\gamma}: \Delta \to S$. Consequently, this establishes that $(S, S \setminus S^o)$ forms a Picard pair. Furthermore, the assertion that $S^o$ is both Borel hyperbolic and Brody hyperbolic follows directly from Proposition \ref{prop_BPT_BH}.
	\end{proof}
    \section{Hyperbolicity properties for Deligne-Mumford stacks}
    For the purposes of this paper, algebraic stacks are all assume to be \emph{of finite type} over an algebraic closed field of characteristic 0, i.e. an algebraic stack is defined as a stack that can be represented as the quotient stack of a groupoid of algebraic spaces $s,t: R \to S$ of finite type over ${\rm Spec}(\bC)$, with smooth structure morphisms $s,t$.      
    Let $\sX$ be an algebraic stack. We denote by $\sX^{\rm an}$ its associated complex analytic stack (\cite{Noohi2005,Behrend2006}).
	\subsection{Hyperbolicity of Deligne-Mumford stacks}
	\begin{defn}
		A reduced Deligne-Mumford stack $\sX$ admitting a quasi-projective coarse moduli space is said to be \emph{schematic Borel hyperbolic} (resp. \emph{schematic Brody hyperbolic}) if for every quasi-finite morphism $X\to\sX$ from a reduced quasi-projective scheme $X$, $X$ is Borel hyperbolic (resp. Borel hyperbolic). 
		
		Let $\mathscr{X}$ be a reduced Deligne-Mumford stack admitting a projective coarse moduli space, and let $\mathscr{Z} \subset \mathscr{X}$ denote a closed substack. The pair $(\mathscr{X}, \mathscr{Z})$ is said to be a \emph{schematic Picard pair} if for every morphism $\eta: X \to \mathscr{X}$ such that $X$ is a reduced projective scheme, the preimage $U := \eta^{-1}(\mathscr{X} \setminus \mathscr{Z})$ is dense in $X$, and the restriction $\eta|_U: U \to \mathscr{X} \setminus \mathscr{Z}$ is a quasi-finite morphism, the pair $(X, X \setminus U)$ forms a Picard pair.
	\end{defn}
	\begin{defn}
		Let $\sX$ be a reduced Deligne-Mumford stack, and let $\sZ \subset \sX$ be a closed substack. The pair $(\sX, \sZ)$ is called a \emph{Picard pair} if the following condition holds: given a holomorphic map $\gamma: \Delta^\ast \to \sX^{\rm an}$, either $0 \in \overline{\gamma^{-1}(\sZ)}$, or there exists some $n$ such that $\gamma \circ \iota_n$ can be extended to a holomorphic map $\overline{\gamma \circ \iota_n}: \Delta \to \sX^{\rm an}$, where $\iota_n: \Delta^\ast \to \Delta^\ast$ is defined by $\iota_n(z) = z^n$.
	\end{defn}
	Let $\mu_n = \{e^{\frac{2\pi\sqrt{-1}k}{n}} : k = 0, 1, \dots, n-1\}$ act on $\Delta$ by multiplication. Then the stacky quotient $[\Delta / \mu_n]$ is a unit disc extending $\Delta^\ast$, equipped with a stacky structure at the origin. The pair $(\sX, \sZ)$ is a \emph{Picard pair} if and only if for every holomorphic map $\gamma: \Delta^\ast \to \sX^{\rm an}$, either $0 \in \overline{\gamma^{-1}(\sZ)}$, or there exists some $n \in \bZ_{>0}$ such that $\gamma$ can be extended to a holomorphic map $\overline{\gamma}: [\Delta / \mu_n] \to \sX^{\rm an}$.
	\begin{defn}\label{defn_log_uniformisable}
		Let $\sX$ be a reduced Deligne-Mumford stack admitting a projective coarse moduli space, and let $\sZ \subset \sX$ be a closed substack. The pair $(\sX, \sZ)$ is said to be \emph{logarithmically uniformisable} if there exists a surjective morphism $\eta: X \to \sX$ such that $X$ is a reduced projective scheme and $\eta$ is a finite \'etale morphism over $\sX \setminus \mathscr{Z}$.
	\end{defn}
    When $\sX$ is a moduli stack, the covering space $X$ introduced in Definition \ref{defn_log_uniformisable} is expected to be a fine moduli space, whose objects are those in $\sX$ equipped with a certain level structure (see Proposition \ref{prop_moduli_curve_uniformisable}).
	\begin{prop}\label{prop_schematic_Picard}
		Let $\mathscr{X}$ be a reduced Deligne-Mumford stack admitting a projective coarse moduli space, and let $\mathscr{Z} \subset \mathscr{X}$ denote a closed substack. Suppose that the pair $(\mathscr{X}, \mathscr{Z})$ is logarithmically uniformisable and forms a schematic Picard pair. Then $(\mathscr{X}, \mathscr{Z})$ is Picard hyperbolic.
	\end{prop}
	\begin{proof}
		Let $\eta: X \to \mathscr{X}$ be a surjective morphism such that $X$ is a reduced projective scheme and $\eta$ is a finite \'etale morphism over $\mathscr{U} := \mathscr{X} \setminus \mathscr{Z}$. Let $U := \eta^{-1}(\mathscr{U})$. Consider a holomorphic map $f: \Delta^* \to \mathscr{X}^{\rm an}$ such that $0 \notin \overline{f^{-1}(\mathscr{Z})}$. By shrinking $\Delta^*$, we may assume that $f(\Delta^*) \subset \mathscr{U}^{\rm an}$. Consider the associated base change diagram:
		$$
		\xymatrix{
			Y \ar[r]^-{g} \ar[d] & U^{\rm an} \subset X^{\rm an} \ar[d]^{\eta} \\
			\Delta^* \ar[r]^-{f} & \mathscr{U}^{\rm an} \subset \mathscr{X}^{\rm an}.
		}
		$$
		Since $Y \to \Delta^*$ is a finite \'etale morphism, $Y$ is a disjoint union of copies of $\Delta^*$, where each component is mapped to $\Delta^*$ via $\iota_n(z) = z^n$ for some positive integer $n$. Selecting one such component yields a morphism $g: \Delta^* \to U^{\rm an}$. Since $(\mathscr{X}, \mathscr{Z})$ is a schematic Picard pair, the pair $(X, X \setminus U)$ forms a Picard pair. Consequently, the morphism $g: \Delta^* \to U^{\rm an}$ extends to a holomorphic map $\overline{g}: \Delta \to X^{\rm an}$. By composing with $\eta$, we obtain the holomorphic extension $\eta \circ \overline{g}: \Delta \to \mathscr{X}^{\rm an}$. This completes the proof.
	\end{proof}
	\begin{defn}
		A reduced Deligne-Mumford stack $\sX$ is said to be \emph{Borel hyperbolic} if every holomorphic map from a reduced algebraic stack to $\sX$ is algebraic.
	\end{defn}
	\begin{defn}
	A reduced Deligne-Mumford stack $\sX$ is called \emph{uniformisable} if there exists a finite \'etale surjective morphism $X\to\sX$, where $X$ is a scheme.
    \end{defn}
	\begin{prop}\label{prop_schematic_Borel}
		Let $\mathscr{X}$ be a uniformisable and reduced Deligne-Mumford stack that admits a quasi-projective coarse moduli space. Suppose that $\mathscr{X}$ is schematically Borel hyperbolic. Then $\mathscr{X}$ is Borel hyperbolic.
	\end{prop} 
	\begin{proof}
		Let $f: X \to \sX$ be a finite \'etale covering by a reduced projective scheme $X$. Let $R = X \times_{\sX} X$, and let $p_1, p_2: R \to X$ denote the two projection maps. Then $\sX \simeq [X/R]$, the stacky quotient. Let $\sY$ be a reduced algebraic stack. We may choose a groupoid algebraic space $s', t': R' \to Y'$ such that $\sY \simeq [Y'/R']$ as stacks. 
		The spaces and morphisms fit into the following commutative diagram:
		$$
		\xymatrix{
			R'^{\rm an} \ar[r]^h \ar@<2pt>[d]\ar@<-2pt>[d] & R^{\rm an} \ar@<2pt>[d]\ar@<-2pt>[d] \\
			Y'^{\rm an} \ar[r]^{g'} \ar[d] & X^{\rm an} \ar[d]\\
			\sY^{\rm an} \ar[r]^g & \sX^{\rm an}.
		}
		$$
		Such an $Y'$ and $g'$ can be constructed as follows. Let $Y'' = X^{\rm an} \times_{\mathscr{X}^{\rm an}} \mathscr{Y}^{\rm an}$, then the base change morphism $Y'' \to \mathscr{Y}^{\rm an}$ is finite \'etale. Hence, by the Riemann extension theorem, $Y''$ is an algebraic stack and the morphism $Y'' \to \mathscr{Y}^{\rm an}$ is algebraic. Now let $Y' \to Y''$ be a smooth surjective morphism from a scheme $Y'$, and we obtain the desired diagram.		
		
		By assumption on $\sX$, $X$ and $R$ are Borel hyperbolic.
		
		{\bf Case I:} Assume that $\sY$ is an algebraic space. Then $Y'$ and $R'$ are reduced schemes of finite type over ${\rm Spec}(\bC)$. By the Borel hyperbolicity of $S$ and $R$, it follows that the holomorphic maps $h$ and $g'$ are algebraic. Consequently, $g$ is algebraic.
		
		{\bf Case II:} Assume that $\sY$ is an algebraic stack. Then $Y'$ and $R'$ are reduced algebraic spaces of finite type over ${\rm Spec}(\bC)$. By the Borel hyperbolicity of $S$ and $R$, and the result established in {\bf Case I}, it follows that the holomorphic maps $h$ and $g'$ are algebraic. Consequently, $g$ is algebraic.
	\end{proof}
    The following Brody hyperbolicity condition for Deligne-Mumford stacks is suggested to me by Javanpeykar.
	\begin{defn}
		Let $\mathscr{X}$ be a reduced Deligne-Mumford stack and let $\iota: \mathscr{X} \to X$ denote the natural map to its coarse moduli space $X$. The stack $\mathscr{X}$ is said to be \emph{Brody hyperbolic} if for any holomorphic map $f: \mathbb{C} \to \mathscr{X}^{\rm an}$, the composition $\iota^{\rm an} \circ f: \mathbb{C} \to X^{\rm an}$ is a constant map.
	\end{defn}
	\begin{prop}\label{prop_schematic_Brody}
		Let $\mathscr{X}$ be a uniformisable and reduced Deligne-Mumford stack that admits a quasi-projective coarse moduli space. Suppose that $\mathscr{X}$ is schematically Brody hyperbolic. Then $\mathscr{X}$ is Brody hyperbolic.
	\end{prop} 
	\begin{proof}
		Let $\iota: \mathscr{X} \to X$ denote the natural map to the coarse moduli space $X$.
		Let $f: S \to \sX$ be a finite \'etale covering by a scheme $S$. Let $R = S \times_{\sX} S$, and let $p_1, p_2: R \to S$ denote the two projection maps. Then $\sX \simeq [S/R]$, the stacky quotient. Consider the base change diagram:
		$$
		\xymatrix{
			R'\ar[r]^h \ar@<2pt>[d]\ar@<-2pt>[d] & R^{\rm an} \ar@<2pt>[d]\ar@<-2pt>[d] \\
			Y'\ar[r]^{g'} \ar[d] & S^{\rm an} \ar[d]\\
			\bC \ar[r]^g & \sX^{\rm an}.
		}
		$$
		Since $Y' \to \bC$ is a finite \'etale morphism, $Y'$ is a disjoint union of copies of $\bC$, where each component is mapped to $\bC$ via the identity map. Selecting one such component yields a morphism $g': \bC \to S^{\rm an}$. Since $S$ is Brody hyperbolic by the assumption that $\sX$ is schematic Brody hyperbolic. Thus $g'$ is a constant map. Consequently, $\iota^{\rm an}\circ g$ is also a constant map.
	\end{proof}
    \subsection{Stratified hyperbolicity of $\sM_{\rm slc}(d,\Phi_c,\Gamma,\sigma)$}\label{section_stack}
    In \cite{stjc2025}, we introduced the concept of a birationally admissible stratification of $\sM_{\rm slc}(d, \Phi_c, \Gamma, \sigma)$. Let $f: (\sX, \sB), \sA \to \sM_{\rm slc}(d, \Phi_c, \Gamma, \sigma)$ denote the universal family of $(d, \Phi_c, \Gamma, \sigma)$-stable minimal models.  
    A \emph{birationally admissible stratification} of $\sM_{\rm slc}(d, \Phi_c, \Gamma, \sigma)$ is a filtration by Zariski-closed, reduced substacks:  
    $$
    \emptyset = S_{-1} \subset S_0 \subset \cdots \subset S_N = \sM_{\rm slc}(d, \Phi_c, \Gamma, \sigma),
    $$  
    such that for each $i = 0, \dots, N$, the pullback family of $f$ to $S_i \setminus S_{i-1}$ is strictly birationally admissible. A stratum is defined as a connected component of $S_i \setminus S_{i-1}$ for some $i \in \{0, \dots, N\}$. Given that the moduli space of smooth varieties may itself lack smoothness, we do not impose the requirement that a birationally admissible stratification must be regular, meaning that every stratum is smooth. 
    
    A birationally admissible stratification always exists but is typically not unique; it is determined by the choice of a functorial desingularization functor (\cite[\S 6]{stjc2025}).
    
    A direct consequence of Theorem \ref{thm_main_proof} is the following:    
    \begin{cor}
    	Let $\sS$ be a stratum of a birationally admissible stratification of $\sM_{\rm slc}(d, \Phi_c, \Gamma, \sigma)$. Then $\sS$ is schematic Picard hyperbolic, schematic Borel hyperbolic, and schematic Brody hyperbolic.
    \end{cor}
    By combining Proposition \ref{prop_schematic_Picard}, Proposition \ref{prop_schematic_Borel}, and Proposition \ref{prop_schematic_Brody}, we obtain the following result.
    \begin{cor}\label{cor_log_uniform_hyperbolic}
    	Let $\sS$ be a stratum of a birationally admissible stratification of $\sM_{\rm slc}(d, \Phi_c, \Gamma, \sigma)$. Suppose that $\sS$ is uniformisable. Then $\sS$ is both Borel hyperbolic and Brody hyperbolic. If $(\overline{\sS}, \overline{\sS} \setminus \sS)$ is logarithmically uniformisable, then $(\overline{\sS}, \overline{\sS} \setminus \sS)$ forms a Picard pair.
    \end{cor}    
    We are interested in the corollaries because of the following proposition.
    \begin{prop}\label{prop_moduli_curve_uniformisable}
    	Let $\sS$ be a stratum of the canonical stratification of $\overline{\sM}_{g,n}$. Then the pair $(\overline{\sS},\overline{\sS}\backslash \sS)$ is logarithmically uniformisable. In particular $\sS$ is uniformisable.
    \end{prop}
    \begin{proof}
    	The desired covering is provided by the fine moduli space of admissible $G$-covers of stable curves. We adopt the notations from \cite[Chapter VXI]{Griffiths2011}. Let $G = G_{p,q}$ denote the group constructed in \cite[Chapter VXI, \S 6]{Griffiths2011}, and let ${_{G}}\overline{M}_{g,n}$ represent the fine moduli space of the relevant admissible $G$-covers. Then, ${_{G}}\overline{M}_{g,n}$ is a smooth projective variety, and there exists a finite surjective morphism $\eta: {_{G}}\overline{M}_{g,n} \to \overline{\sM}_{g,n}$.
    	
    	This morphism $\eta$ admits the following local description (\cite[Remark (9.4)]{Griffiths2011}). Consider a point $p \in {_{G}}\overline{M}_{g,n}$, and let $B$ be an analytic neighborhood of $\eta(p)$ in $\overline{\sM}_{g,n}$ (called the "base of the Kuranishi family" of the marked stable curve representing $\eta(p)$ in loc. cit.), equipped with holomorphic coordinates $z_1, \dots, z_m$, where $m = 3g - 3 + n > 0$. In this coordinate system, the boundary $\partial\sM_{g,n}$ is defined by the equation $z_1 \cdots z_r = 0$, and the subspace $\overline{\sS}$ is defined by $z_a = z_{a+1} = \cdots = z_b = 0$ for some indices satisfying $1 \leq a \leq b \leq m$. There exists an analytic neighborhood $\widetilde{B}$ of $p$ in ${_{G}}\overline{M}_{g,n}$, together with holomorphic coordinates $w_1, \dots, w_m$, such that the morphism $\eta$ is given by
    	$$\eta^\ast(w_1)=z_1^k,\quad\dots\quad, \eta^\ast(w_r)=z_r^k,\quad \eta^\ast(w_{r+1})=z_{r+1}, \quad\dots\quad, \eta^\ast(w_{m})=z_{m}$$
    	for some integer $k \geq 1$.
    	
    	From this local description, it follows that $\overline{S} := \eta^{-1}(\overline{\sS})_{\rm red}$ is a smooth projective variety, and the restriction $\eta|_{\overline{S}}: \overline{S} \to \overline{\sS}$ is proper and \'etale over $\sS$. Consequently, the pair $(\overline{\sS}, \overline{\sS} \setminus \sS)$ is logarithmically uniformisable.
    \end{proof}
    As a direct consequence, we obtain the following result.    
    \begin{cor}
    	Let $\mathscr{S}$ be a stratum of the canonical stratification of $\overline{\mathscr{M}}_{g,n}$. Then the pair $(\overline{\mathscr{S}}, \overline{\mathscr{S}} \setminus \mathscr{S})$ is a Picard pair. Moreover, $\mathscr{S}$ is both Borel hyperbolic and Brody hyperbolic.
    \end{cor}
    Motivated by Proposition \ref{prop_moduli_curve_uniformisable} and Conjecture 1.6 of Javanpeykar-Sun-Zuo \cite{JSZ2024}, we make the following conjecture.
    \begin{conj}
    	Let $\sS$ be a stratum of a birationally admissible stratification of $\sM_{\rm slc}(d, \Phi_c, \Gamma, \sigma)$. Then the pair $(\overline{\sS},\overline{\sS}\backslash \sS)$ is logarithmically uniformisable.
    \end{conj}
	\bibliographystyle{plain}
	\bibliography{Hyper_Moduli}
\end{document}